\documentclass[a4paper, 11pt]{amsart}
\usepackage{a4wide}
\usepackage[english]{babel} 
\usepackage{amsmath, amssymb} 
\usepackage[all,cmtip]{xy} 
\usepackage{tikz}
\usepackage{tikz-cd}
\usepackage{xcolor}
\usepackage[normalem]{ulem}
\usepackage[mathscr]{euscript}
\usepackage{pifont}
\usepackage{hyperref}
\usepackage[capitalize,noabbrev]{cleveref}
\usepackage{dsfont}
\usepackage{stmaryrd}

\DeclareFontFamily{T1}{cbgreek}{}
\DeclareFontShape{T1}{cbgreek}{m}{n}{<-6>  grmn0500 <6-7> grmn0600 <7-8> grmn0700 <8-9> grmn0800 <9-10> grmn0900 <10-12> grmn1000 <12-17> grmn1200 <17-> grmn1728}{}
\DeclareSymbolFont{quadratics}{T1}{cbgreek}{m}{n}
\DeclareMathSymbol{\qoppa}{\mathord}{quadratics}{19}
\DeclareMathSymbol{\Qoppa}{\mathord}{quadratics}{21}

\newcommand{\Sp}{\mathrm{Sp}}
\newcommand{\GW}{\mathrm{GW}}
\newcommand{\K}{\mathrm{K}}
\newcommand{\s}{\mathrm{s}}
\newcommand{\gs}{\mathrm{gs}}
\newcommand{\gq}{\mathrm{gq}}
\newcommand{\gl}{\mathrm{g}\lambda}
\newcommand{\q}{\mathrm{q}}

\renewcommand{\L}{\mathrm{L}}

\newcommand{\Dperf}{\mathscr{D}^{\mathrm{p}}}

\newcommand{\burn}{\mathrm{b}}

\renewcommand{\ge}{\mathrm{ge}}
\newcommand{\N}{\mathrm{N}}

\newcommand{\QF}{\Qoppa}

\newcommand{\Z}{\mathbb{Z}}

\renewcommand{\SS}{\mathbb{S}}

\newcommand{\map}{\mathrm{map}}

\newcommand{\m}{\mathfrak{m}}
\newcommand{\Mod}{\mathrm{Mod}}
\newcommand{\Fun}{\mathrm{Fun}}
\newcommand{\lto}{\longrightarrow}

\newtheorem*{ThmA}{Theorem A}

\theoremstyle{definition}

\newtheorem{Rmk}{Remark}

\theoremstyle{remark}

\newtheorem*{claim*}{Claim}

\theoremstyle{plain}
\newcounter{zaehler}

\newtheorem*{introcor*}{Corollary}

\title{Gabber rigidity in hermitian K-theory}

\author{Markus Land}
\address{Mathematisches Institut, Ludwig-Maximilians-Universit\"at M\"unchen, Theresienstra\ss e 39, 80333 M\"unchen, Germany}
\email{markus.land@math.lmu.de}

\thanks{The author was supported by the Danish National Research Foundation through the Copenhagen Centre for Geometry and Topology (DNRF151).}

\date{\today}

\begin{document}

\begin{abstract}
We note that Gabber's rigidity theorem for the algebraic K-theory of henselian pairs also holds true for hermitian K-theory with respect to arbitrary form parameters. 
\end{abstract}

\maketitle

Let $R$ be a commutative ring and $\m \subseteq R$ an ideal such that $(R,\m)$ is a henselian pair. Standard examples include henselian local rings like valuation rings of complete nonarchimedean fields as well as pairs where $R$ is $\m$-adically complete or where $\m$ is locally nilpotent. We write $F = R/\m$ and let $n$ be a natural number which is invertible in $R$.
Then Gabber's rigidity theorem \cite{Gabber} says that the canonical map 
\[ K(R)/n \lto K(F)/n \]
is an equivalence; this result was preceded by work of Suslin \cite{Suslin} who showed this conclusion for henselian valuation rings. See also \cite{CMM} for an extension of this result, involving topological cyclic homology, to the case where $n$ need not be invertible in $R$ and a general discussion of henselian pairs. The purpose of this short note is to use the results of \cite{CDHII,CDHIII} as well as \cite{HNS} to show that Gabber's rigidity property also holds true for hermitian K-theory, a.k.a.\ Grothendieck--Witt theory. 
\newline

To state the main result, let $\lambda$ be a form parameter over $R$ in the sense of \cite[\S 3]{Schlichting}, see also \cite[Definition 4.2.26]{CDHI}. In loc.\ cit.\ it is explained that such a form parameter $\lambda$ is equivalently described by a Poincar\'e structure $\QF^{\gl}_R$ in the sense of \cite{CDHI} on $\Dperf(R)$ which sends projective $R$-modules to discrete spectra. Here, $\Dperf(R)$ denotes the stable $\infty$-category of perfect complexes over $R$. 
We will assume that the $\Z$-module with involution over $R$ underlying the form parameter $\lambda$ by is given by $\pm R$, that is, given by the $R$-module $R$ with $C_2$-action either the identity or multiplication by $-1$, viewed as an $R\otimes R$-module via the multiplication map.
There is then an induced form parameter on $F$ whose associated Poincar\'e structure on $\Dperf(F)$ we will denote by $\QF^{\gl}_F$, see \cref{Remark} below for details. The construction is made so that the extension of scalars functor canonically refines to a Poincar\'e functor $(\Dperf(R),\QF^{\gl}_R) \to (\Dperf(F),\QF^{\gl}_F)$ and therefore a map on Grothendieck--Witt theory. Standard examples of form parameters capture the notion of quadratic, even, and symmetric forms (as well as their skew-quadratic, skew-even, and skew-symmetric cousins) with associated Poincar\'e structures $\QF^{\pm \gq}, \QF^{\pm \ge},$ and $\QF^{\pm \gs}$. A further example is provided by the Burnside Poincar\'e structure $\QF^\burn$ whose L-theory was calculated explicitly for $\Z$ in \cite[Example 1.3.18]{CDHIII} and whose $0$'th Grothendieck--Witt group was studied for commutative rings with 2 invertible in the PhD thesis of Dylan Madden \cite{Madden}.
With this notation fixed, we have the following result.

\begin{ThmA}
Let $(R,\m)$ be a henselian pair, $F=R/\m$ and let $n$ be a natural number invertible in $R$. Then the canonical map 
\[ \GW(R;\QF^{\gl}_R)/n \lto \GW(F;\QF^{\gl}_F)/n \]
is an equivalence.
\end{ThmA}
\begin{proof}
The main result of \cite{CDHII} gives a diagram of horizontal fibre sequences
\[\begin{tikzcd}
	\K(R)_{hC_2} \ar[r] \ar[d] & \GW(R;\QF^{\gl}_R) \ar[r] \ar[d] & \L(R;\QF^{\gl}_R) \ar[d] \\
	\K(F)_{hC_2} \ar[r] & \GW(F;\QF^{\gl}_F) \ar[r] & \L(F;\QF^{\gl}_F)
\end{tikzcd}\]
and by Gabber rigidity, the left vertical map becomes an equivalence after tensoring with $\SS/n$. Therefore, the statement of the theorem is equivalent to the statement that the map 
\[ \L(R;\QF^{\gl}_R)/n \lto \L(F;\QF^{\gl}_R)/n \]
is an equivalence. 
We then consider the diagram
\[ \begin{tikzcd}
	\L(R;\QF^{\q}_{\pm R}) \ar[r] \ar[d] & \L(R;\QF^{\gl}_R) \ar[d] \\
	\L(F;\QF^{\q}_{\pm F}) \ar[r] & \L(F;\QF^{\gl}_F)
\end{tikzcd}\]
where $\QF^{\q}_{\pm R}$ denotes the homotopy quadratic Poincar\'e structure associated to the invertible module with involution $M$ which is part of the form parameter $\lambda$, and likewise for $\QF^{\q}_{\pm F}$. We now observe that the formula for relative L-theory obtained in \cite{HNS} shows that the top and bottom horizontal cofibre are $\SS[\tfrac{1}{n}]$-modules.

Indeed, \cite{HNS} shows that the cofibre of the top horizontal arrow is a filtered colimit of objects of the form
\[ \mathrm{Eq} \big( \map_{R}(T\otimes_R T,R) \rightrightarrows (\Sigma^{1-\sigma}\map_{R}(T\otimes_R T,R))_{hC_2} \big) \]
for $T \in \Dperf(R)$. Since $\map_{R}(T\otimes_R T,R)$ is canonically an $R$-module and $n$ is invertible in $R$, it is also an $\SS[\tfrac{1}{n}]$-module. Moreover, since $\Mod(\SS[\tfrac{1}{n}]) \subseteq \Sp$ is a full subcategory closed under colimits and limits, both terms in the equaliser, and therefore also the equaliser itself belong to $\Mod(\SS[\tfrac{1}{n}])$. 
Consequently, the horizontal maps in the above diagram become equivalences upon tensoring with $\SS/n$. The statement of the main theorem is therefore equivalent to the statement that the left vertical map in the above commutative square is an equivalence. This is a consequence of work of Wall's \cite{Wall} as explained in \cite[Prop.\ 2.3.7 \& Remark 2.3.8]{CDHIII}.
\end{proof}

\begin{Rmk}
Restricting the situation above to form parameters rather than general Poincar\'e structures on $\Dperf(R)$ was merely a cosmetic choice to obtain a result about classical Grothendieck--Witt theory: 
Indeed, it is again a consequence of the main theorem of \cite{CDHII} that the diagram
\[\begin{tikzcd}
	\GW(R;\QF^\q_{\pm R}) \ar[r] \ar[d] & \GW(R;\QF) \ar[d] \\
	\L(R;\QF^\q_{\pm R}) \ar[r] & \L(R;\QF)
\end{tikzcd}\]
is a pullback diagram for any Poincar\'e structure $\QF$ on $\Dperf(R)$ whose $\Z$-module with involution over $R$ is given by $\pm R$. The proof presented above therefore shows that for any ring $R$ in which $n$ is invertible, the canonical map 
\[ \GW(R;\QF_{\pm R}^\q)/n \lto \GW(R;\QF)/n \]
is an equivalence and that Gabber rigidity holds for the Poincar\'e structure $\QF^{\q}_{\pm R}$.

In particular, Gabber rigidity also applies to the homotopy symmetric Poincar\'e structure $\QF^{\pm\s}$ as well as the Tate Poincar\'e structure $\QF^{\mathrm{t}}_R$, see \cite[Example 3.2.12]{CDHI}.
\end{Rmk}

\begin{Rmk}
Rigidity in hermitian K-theory has of course been studied in several works before, see for instance \cite{Karoubi, Jardine, HY, Yagunov} for the case of rings with involution. The main purpose here is to show how to use the formalism of Poincar\'e categories and the main result of \cite{CDHII} to reduce rigidity in hermitian K-theory to rigidity in algebraic K-theory and L-theory in a way that allows to treat general form parameters. 
\end{Rmk}

\begin{Rmk}\label{Remark}
In this remark, we describe how extension of scalars can be used to prolong a form parameter over $R$ along a map $R \to R'$ of rings.
It is here that the assumption on the underlying module with involution is used. Indeed, we will describe a general construction on Hermitian structures, and the assumption is used to ensure that the given Poincar\'e structure is sent to a Poincar\'e structure rather than merely a Hermitian structure.

Namely, in \cite[\S 3.3]{CDHI}, we have shown that the category of Hermitian structures on $\Dperf(R)$ is equivalent to the category $\Mod_{\N(R)}(\Sp^{C_2}) = \Mod(\N(R))$, that is, the category of modules over the multiplicative norm $\N(R)$ in the category $\Sp^{C_2}$ of genuine $C_2$-spectra. Moreover, the category $\Mod(\N(R))$ is equipped with a canonical $t$-structure whose heart is equivalent to the category of (possibly degenerate) form parameters over $R$, see \cite[Remark 4.2.27]{CDHI}. Objects in $\Mod(\N(R))$
are described by triples $(M,N, \alpha)$ where
\begin{enumerate}
\item[-] $M$ is an object of $\Mod_{R\otimes R}(\Sp^{BC_2})$, where $R\otimes R$ is an algebra in spectra with $C_2$-action where the action flips the two tensor factors,
\item[-] $N$ is an object of $\Mod(R)$, and
\item[-] $\alpha$ is a map $N \to M^{tC_2}$ of $R$-modules,
\end{enumerate}
see \cite{CDHI}; the Poincar\'e structures then consist of the above triples where $M$ is \emph{invertible} in the sense of \cite[Def.\ 3.1.3]{CDHI}. We warn the reader that caution has to be taken in regards to how $M^{tC_2}$ is to be viewed as an $R$-module, see e.g.\ \cite[pg.\ 7]{CDHIII} for the details. An object $(M,N, \alpha)$ is connective in the canonical $t$-structure on $\Mod(\N(R))$ if and only if $M$ and $N$ are connective. 

The Poincar\'e structure associated to the triple $(M,N,\alpha)$ is denoted by $\QF^\alpha_M$.
Assuming that $M$ is in the image of the canonical functor $\Fun(BC_2,\Mod(R)) \to \Mod_{R\otimes R}(\Fun(BC_2,\Sp))$, the triple 
\[ (M',N', \alpha') = (R'\otimes_R M, R'\otimes_R N, R'\otimes_R N \to R'\otimes_R M^{tC_2} \to (R'\otimes_R M)^{tC_2}) \]
gives rise to a Poincar\'e structure on $\Dperf(R')$ for which the extension of scalar functor canonically refines to a Poincar\'e functor $(\Dperf(R),\QF^\alpha_M) \to (\Dperf(R'),\QF^{\alpha'}_{M'})$, see \cite[Lemma 3.4.3]{CDHI}. Now, if $(M,N,\alpha)$ was associated to a form parameter, then the same need not be true for the triple $(M',N',\alpha')$: Indeed, this is the case if and only if $\QF^{\alpha'}_{M'}(R')$ is a discrete spectrum which in general need not be the case. However, we may consider the composite
\[ M'_{hC_2} \lto \QF^{\alpha'}_{M'}(R') \lto \tau_{\leq 0}\QF^{\alpha'}_{M'}(R') \]
and denote its cofibre by $N''$. The pushout diagram of spectra
\[\begin{tikzcd}
	\QF^{\alpha'}_{M'}(R') \ar[r] \ar[d] & N' \ar[d] \\
	(M')^{hC_2} \ar[r] & (M')^{tC_2} 
\end{tikzcd}\]
and the fact that $(M')^{hC_2}$ is coconnective shows that there is a canonical map $\alpha''\colon N'' \to (M')^{tC_2}$. By construction, the triple $(M', N'', \alpha'')$ is an object of $\Mod(\N(R'))^{\heartsuit}$ and in fact identifies with $\tau_{\leq 0}(M',N',\alpha')$. This object determines a Poincar\'e structure $\QF^{\gl'}$ associated to a form parameter $\lambda'$ over $R'$ for which the extension of scalars functor refines to a Poincar\'e functor
\[ (\Dperf(R),\QF^{\gl}) \lto (\Dperf(R'),\QF^{\gl'}).\]
%
%
%

To give an example of this construction, we recall the genuine Poincar\'e structures $\QF^{\geq m}_{\pm R}$ which, for $m=0,1,2$ are the Poincar\'e structures $\QF^{\gq}_{\pm R}$, $\QF^{\ge}_{\pm R}$ and $\QF^{\gs}_{\pm R}$ associated to the classical (skew-) quadratic, even, and symmetric form parameter over $R$, respectively, see \cite[Remark R.3 \& R.5]{CDHIII}. In this case, the extension of scalars functor associated to a ring map $R \to R'$ indeed sends $\QF^{\geq m}_{\pm R}$ to $\QF^{\geq m}_{\pm R'}$.
\end{Rmk}

\begin{Rmk}
For the Poincar\'e structures $\QF^{\geq m}_{\pm R}$ one can give the following argument that the map 
\[ \GW(R;\QF^{\geq m}_{\pm R})/n \lto \GW(F;\QF^{\geq m}_{\pm R})/n \]
is an equivalence without appealing to the general formula for relative L-theory of \cite{HNS}. Namely, in \cite[Prop.\ 3.1.14]{CDHIII} we have shown that the map
\[ \L(R;\QF^{\q}_{\pm R})[\tfrac{1}{2}] \lto \L(R;\QF^{\geq m}_{\pm R})[\tfrac{1}{2}] \]
is an equivalence for all $m \in \Z$. Therefore, the proof of the theorem applies in the case where 2 does not divide $n$. In the case where $2$ divides $n$, we deduce that $2$ is invertible in $R$ in which case already the map $\QF^{\q}_{\pm R} \to \QF^{\geq m}_{\pm R}$ is an equivalence of Poincar\'e structures, see \cite[Remark R.4]{CDHIII}.
\end{Rmk}

\begin{Rmk}
Suppose that $R$ is an associative ring which is $\m$-adically complete for an ideal $\m \subset R$. Then the result of Wall, see again \cite[Prop.\ 2.3.7]{CDHIII}, says that the map $\L^{\pm \q}(R) \to \L^{\pm \q}(R/\m)$ is an equivalence. To the best of our knowledge, it is not known whether also the map $K(R)/n \to K(R/\m)/n$ is an equivalence. However, if it is, this argument shows that the same is true for Grothendieck--Witt theory and vice versa.
\end{Rmk}

\bibliographystyle{amsalpha}
\bibliography{rigidity}

\newcommand{\etalchar}[1]{$^{#1}$}
\providecommand{\bysame}{\leavevmode\hbox to3em{\hrulefill}\thinspace}
\providecommand{\MR}{\relax\ifhmode\unskip\space\fi MR }
\providecommand{\MRhref}[2]{%
  \href{http://www.ams.org/mathscinet-getitem?mr=#1}{#2}
}
\providecommand{\href}[2]{#2}
\begin{thebibliography}{CDH{\etalchar{+}}20b}

\bibitem[CDH{\etalchar{+}}20a]{CDHI}
B.~Calm{\'e}s, E.~Dotto, Y.~Harpaz, F.~Hebestreit, M.~Land, K.~Moi, D.~Nardin,
  T.~Nikolaus, and W.~Steimle, \emph{Hermitian {K}-theory for stable
  $\infty$-categories {I}: {F}oundations}, arXiv:2009.07223 (2020).

\bibitem[CDH{\etalchar{+}}20b]{CDHII}
\bysame, \emph{Hermitian {K}-theory for stable $\infty$-categories {II}:
  {Cobordism categories and additivity}}, arXiv:2009.07224 (2020).

\bibitem[CDH{\etalchar{+}}20c]{CDHIII}
\bysame, \emph{Hermitian {K}-theory for stable $\infty$-categories {III}:
  {Grothendieck--Witt groups of rings}}, arXiv:2009.07225 (2020).

\bibitem[CMM21]{CMM}
D.~Clausen, A.~Mathew, and M.~Morrow, \emph{{$K$}-theory and topological cyclic
  homology of henselian pairs}, J. Amer. Math. Soc. \textbf{34} (2021), no.~2,
  411--473.

\bibitem[Gab92]{Gabber}
O.~Gabber, \emph{{$K$}-theory of {H}enselian local rings and {H}enselian
  pairs}, Algebraic {$K$}-theory, commutative algebra, and algebraic geometry
  ({S}anta {M}argherita {L}igure, 1989), Contemp. Math., vol. 126, Amer. Math.
  Soc., Providence, RI, 1992, pp.~59--70.

\bibitem[HNS22]{HNS}
Y.~Harpaz, T.~Nikolaus, and J.~Shah, \emph{{Real topological cyclic homology
  and normal L-theory}}, in preparation, 2022.

\bibitem[HY07]{HY}
J.~Hornbostel and S.~Yagunov, \emph{Rigidity for {H}enselian local rings and
  {$\Bbb A^1$}-representable theories}, Math. Z. \textbf{255} (2007), no.~2,
  437--449.

\bibitem[Jar83]{Jardine}
J.~F. Jardin, \emph{A rigidity theorem for {$L$}-theory}, preprint, University
  of Chicago, 1983.

\bibitem[Kar84]{Karoubi}
M.~Karoubi, \emph{Relations between algebraic {$K$}-theory and {H}ermitian
  {$K$}-theory}, Proceedings of the {L}uminy conference on algebraic
  {$K$}-theory ({L}uminy, 1983), vol.~34, 1984, pp.~259--263.

\bibitem[Mad21]{Madden}
D.~Madden, \emph{Burnside form rings and the {K}-theory of forms}, Ph.D.
  thesis, University of Warwick, 2021.

\bibitem[Sch21]{Schlichting}
M.~Schlichting, \emph{Higher {$K$}-theory of forms {I}. {F}rom rings to exact
  categories}, J. Inst. Math. Jussieu \textbf{20} (2021), no.~4, 1205--1273.

\bibitem[Sus84]{Suslin}
A.~A. Suslin, \emph{On the {$K$}-theory of local fields}, Proceedings of the
  {L}uminy conference on algebraic {$K$}-theory ({L}uminy, 1983), vol.~34,
  1984, pp.~301--318.

\bibitem[Wal73]{Wall}
C.~T.~C. Wall, \emph{On the classification of {H}ermitian forms. {III}.
  {C}omplete semilocal rings}, Invent. Math. \textbf{19} (1973), 59--71.

\bibitem[Yag22]{Yagunov}
S.~Yagunov, \emph{{Suslin's work on the {K}-theory of local fields.
  Revisited}}, arXiv:2201.00728 (2022).

\end{thebibliography}

\end{document}